\documentclass[11pt]{amsart}

\newtheorem{theorem}{Theorem}[section]
\newtheorem{proposition}[theorem]{Proposition}
\newtheorem{lemma}[theorem]{Lemma}

\newtheorem{fact}[theorem]{Fact}
\newtheorem{statement}[theorem]{Statement}
\newtheorem{cor}[theorem]{Corollary}
 
\newtheorem{definition}[theorem]{Definition}

\theoremstyle{plain}
\numberwithin{equation}{theorem}

\theoremstyle{remark}
\newtheorem{remark}[theorem]{Remark}

\newcommand{\C}{{\mathbb C}}

\newcommand{\Kbar}{\overline{K}}

\newcommand{\tensor}{\otimes}

\DeclareMathOperator{\Frac}{Frac}

\DeclareMathOperator{\sep}{sep}
\DeclareMathOperator{\tor}{tor}

\DeclareMathOperator{\CC}{\mathbb{C}_{\infty}}

\newcommand{\bP}{{\mathbb P}}

\newcommand{\bG}{{\mathbb G}}

\newcommand{\bC}{{\mathbb C}}

\newcommand{\bF}{{\mathbb F}}
\newcommand{\Fq}{\bF_q}
\newcommand{\lra}{\longrightarrow}

\newcommand{\cL}{\mathcal{L}}
\newcommand{\cT}{\mathcal{T}}
\newcommand{\cF}{\mathcal{F}}

\newcommand{\bK}{\overline{K}}
\newcommand{\hhat}{{\widehat h}}

\newcommand{\Drin}{{\bf \phi}}

\title{Siegel's theorem for Drinfeld modules}
\author{D.~Ghioca and T.~J.~Tucker}


\address{
Dragos Ghioca \\
Department of Mathematics\\
McMaster University \\
1280 Main Street West \\ 
Hamilton, Ontario \\
Canada  L8S 4K1 \\
}

\email{dghioca@math.mcmaster.ca}

\address{
Thomas Tucker\\
Department of Mathematics\\
Hylan Building\\
University of Rochester\\
Rochester, NY 14627 \\
}

\email{ttucker@math.rochester.edu}

\begin{document}

\begin{abstract}
  We prove a Siegel type statement for finitely generated $\phi$-submodules of $\mathbb{G}_a$ under the action of a Drinfeld module $\phi$. This provides a positive answer to a question we asked in a previous paper. We also prove an analog for Drinfeld modules of a theorem of Silverman for nonconstant rational maps of $\bP^1$ over a number field.
\end{abstract}

\maketitle

\section{Introduction}\label{intro} 
\footnotetext[1]{dghioca@math.mcmaster.ca; ttucker@math.rochester.edu}

In 1929, Siegel (\cite{Siegel}) proved that if $C$ is an irreducible
affine curve defined over a number field $K$ and $C$ has at least
three points at infinity, then there are at most finitely many
$K$-rational points on $C$ that have integral coordinates. The proof
of this famous theorem uses diophantine approximation along with the
fact that certain groups of rational points are finitely generated;
when $C$ has genus greater than 0, the group in question is the
Mordell-Weil group of the Jacobian of $C$, while when $C$ has genus 0,
the group in question is the group of $S$-units in a finite extension
of $K$.

Motivated by the analogy between rank $2$ Drinfeld modules and
elliptic curves, the authors conjectured in \cite{findrin} a Siegel
type statement for finitely generated $\phi$-submodules $\Gamma$ of
$\bG_a$ (where $\phi$ is a Drinfeld module of arbitrary rank). For a
finite set of places $S$ of a function field $K$, we defined a notion
of $S$-integrality and asked whether or not it is possible that there
are infinitely many $\gamma\in\Gamma$ which are $S$-integral with
respect to a fixed point $\alpha\in\bK$. We also proved in
\cite{findrin} a first instance of our conjecture in the case where
$\Gamma$ is a cyclic submodule and $\alpha$ is a torsion point for
$\phi$. Our goal in this paper is to prove our Siegel conjecture for
every finitely generated $\phi$-submodule of $\mathbb{G}_a(K)$, where
$\phi$ is a Drinfeld module defined over the field $K$ (see our
Theorem~\ref{Siegel}). We will also establish an analog (also in the
context of Drinfeld modules) of a theorem of Silverman for nonconstant
morphisms of $\bP^1$ of degree greater than $1$ over a number field
(see our Theorem~\ref{Silverman}).

We note that recently there has been significant progress on
establishing additional links between classical diophantine results
over number fields and similar statements for Drinfeld modules. Denis
\cite{Denis-conjectures} formulated analogs for Drinfeld modules of
the Manin-Mumford and the Mordell-Lang conjectures.  The
Denis-Manin-Mumford conjecture was proved by Scanlon in
\cite{Scanlon}, while a first instance of the Denis-Mordell-Lang
conjecture was established in \cite{IMRN} by the first author (see
also \cite{full-ml-drinfeld} for an extension of the result from
\cite{IMRN}). The authors proved in \cite{dynml} several other cases
of the Denis-Mordell-Lang conjecture. In addition, the first author
proved in \cite{Mat.Ann} an equidistribution statement for torsion
points of a Drinfeld module that is similar to the equidistribution
statement established by Szpiro-Ullmo-Zhang \cite{suz} (which was
later extended by Zhang \cite{Zhang} to a full proof of the famous
Bogomolov conjecture). Breuer \cite{Breuer} proved a special case of
the Andr\'{e}-Oort conjecture for Drinfeld modules, while special
cases of this conjecture in the classical case of a number field were
proved by Edixhoven-Yafaev \cite{Edixhoven} and Yafaev \cite{Yafaev}.
Bosser \cite{Bosser} proved a lower bound for linear forms in
logarithms at an infinite place associated to a Drinfeld module
(similar to the classical result obtained by Baker \cite{baker} for
usual logarithms, or by David \cite{david} for elliptic logarithms).
Bosser's result was used by the authors in \cite{findrin} to establish
certain equidistribution and integrality statements for Drinfeld
modules.  Moreover, Bosser's result is believed to be true also for
linear forms in logarithms at finite places for a Drinfeld module (as
was communicated to us by Bosser). Assuming this last statement, we
prove in this paper the natural analog of Siegel's theorem for
finitely generated $\phi$-submodules. We believe that our present
paper provides additional evidence that the Drinfeld modules represent
a good arithmetic analog in characteristic $p$ for abelian varieties
in characteristic $0$.

The basic outline of this paper can be summarized quite briefly. In
Section~\ref{notation} we give the basic definitions and notation, and
then state our main results. In Section~\ref{proof} we prove these
main results: Theorems~\ref{Siegel} and \ref{Silverman}.

\section{Notation}\label{notation}

{\bf Notation.} $\mathbb{N}$ stands for the non-negative integers: $\{0,1,\dots\}$, while $\mathbb{N}^{*}:=\mathbb{N}\setminus\{0\}$ stands for the positive integers.

\subsection{Drinfeld modules}
We begin by defining a Drinfeld module.  Let $p$ be a prime and let
$q$ be a power of $p$. Let $A:=\mathbb{F}_q[t]$, let $K$ be a finite field extension of $\mathbb{F}_q(t)$, and let $\Kbar$ be an
algebraic closure of $K$. We let $\tau$ be the Frobenius on
$\mathbb{F}_q$, and we extend its action on $\Kbar$.  Let $K\{\tau\}$
be the ring of polynomials in $\tau$ with coefficients from $K$ (the
addition is the usual addition, while the multiplication is the
composition of functions).

A Drinfeld module is a morphism $\Drin:A\rightarrow K\{\tau\}$ for
which the coefficient of $\tau^0$ in $\Drin(a)=:\Drin_a$ is $a$ for
every $a\in A$, and there exists $a\in A$ such that $\Drin_a\ne
a\tau^0$. The definition given here represents what Goss \cite{Goss}
calls a Drinfeld module of ``generic characteristic''.

We note that usually, in the definition of a Drinfeld module, $A$ is
the ring of functions defined on a projective nonsingular curve $C$,
regular away from a closed point $\eta\in C$. For our definition of a
Drinfeld module, $C=\mathbb{P}^1_{\mathbb{F}_q}$ and $\eta$ is the
usual point at infinity on $\mathbb{P}^1$. On the other hand, every
ring of regular functions $A$ as above contains $\mathbb{F}_q[t]$ as a
subring, where $t$ is a nonconstant function in $A$.
 
For every field extension $K\subset L$, the Drinfeld module $\Drin$
induces an action on $\mathbb{G}_a(L)$ by $a*x:=\Drin_a(x)$, for each
$a\in A$. We call \emph{$\phi$-submodules} subgroups of
$\mathbb{G}_a(\overline{K})$ which are invariant under the action of
$\phi$. We define the \emph{rank} of a $\phi$-submodule $\Gamma$ be 
$$\dim_{\Frac(A)}\Gamma\tensor_A\Frac(A).$$
As shown in \cite{Poonen}, $\mathbb{G}_a(K)$
is a direct sum of a finite torsion $\phi$-submodule with a free
$\phi$-submodule of rank $\aleph_0$.

A point $\alpha$ is \emph{torsion} for the Drinfeld module action if
and only if there exists $Q\in A\setminus\{0\}$ such that
$\Drin_Q(\alpha)=0$. The monic polynomial $Q$ of minimal degree which
satisfies $\phi_Q(\alpha)=0$ is called the \emph{order} of $\alpha$.
Since each polynomial $\phi_Q$ is separable, the torsion submodule
$\phi_{\tor}$ lies in the separable closure $K^{\sep}$ of $K$.

\subsection{Valuations and Weil heights}
Let $M_{\mathbb{F}_q(t)}$ be the set of places on $\Fq(t)$.
We denote by $v_{\infty}$ the place in $M_{\Fq(t)}$ such that $v_{\infty}(\frac{f}{g})=\deg(g)-\deg(f)$ for
every nonzero $f,g\in A=\Fq[t]$. We let $M_K$ be the set of valuations on $K$. Then $M_K$ is a set of valuations which satisfies a
product formula (see \cite[Chapter 2]{Serre-Mordell_Weil}). Thus
\begin{itemize}
\item for each nonzero $x\in K$, there are finitely many $v\in M_K$
such that $|x|_v\ne 1$; and \\
\item for each nonzero $x\in K$, we have $\prod_{v\in M_K} |x|_v=1$.
\end{itemize}
We may use these valuations to define a
Weil height for each $x \in K$ as
\begin{equation}\label{weil}
h(x) = \sum_{v \in M_K} \max \log (|x|_v,1).
\end{equation}

{\bf Convention.} Without loss of generality we may assume that the normalization for all the valuations of $K$ is made so that for each $v\in M_K$, we have $\log |x|_v\in\mathbb{Z}$.

\begin{definition}
Each place in $M_K$ which lies over $v_{\infty}$ is called an infinite place. Each place in $M_K$ which does not lie over $v_{\infty}$ is called a finite place.
\end{definition}

\subsection{Canonical heights}
Let $\Drin:A\rightarrow K\{\tau\}$ be a Drinfeld module of \emph{rank}
$d$ (i.e. the degree of $\phi_t$ as a polynomial in $\tau$ equals
$d$). The canonical height of $\beta\in K$ relative to $\Drin$ (see \cite{Denis}) is
defined as
$$\hhat(\beta) = \lim_{n \to \infty}
\frac{h(\Drin_{t^n}(\beta))}{q^{nd}}.$$ 
Denis \cite{Denis} showed that a point is torsion if and only if its canonical height equals $0$. 

For every $v\in M_K$, we let the local canonical height of $\beta\in K$ at $v$ be
\begin{equation}\label{poon-def}
\hhat_v(\beta) = \lim_{n \to \infty} \frac{\log
  \max(|\Drin_{t^n}(\beta)|_v, 1)}{q^{nd}}.
\end{equation}
Furthermore, for every $a\in \Fq[t]$, we have $\hhat_v(\phi_a(x))=\deg(\phi_a)\cdot\hhat_v(x)$ (see \cite{Poonen}).
It is clear that $\hhat_v$ satisfies the triangle inequality, and also that $\sum_{v\in M_K} \hhat_v(\beta) = \hhat(\beta)$.

\subsection{Completions and filled Julia sets}

By abuse of notation, we let $\infty\in M_K$ denote any place extending the
place $v_{\infty}$.  We let $K_{\infty}$ be
the completion of $K$ with respect to $|\cdot|_{\infty}$. We let
$\overline{K_{\infty}}$ be an algebraic closure of $K_{\infty}$. We
let $\CC$ be the completion of $\overline{K_{\infty}}$. Then $\CC$ is
a complete, algebraically closed field. Note that $\CC$ depends on our
choice for $\infty\in M_K$ extending $v_{\infty}$. However, each time
we will work with only one such place $\infty$, and so, there will be
no possibility of confusion.

Next, we define the \emph{$v$-adic filled Julia set} $J_{\phi,v}$
corresponding to the Drinfeld module $\phi$ and to each place $v$ of
$M_K$.  Let $\mathbb{C}_v$ be the completion of an algebraic closure
of $K_v$.  Then $| \cdot |_v$ extends to a unique absolute value on
all of $\mathbb{C}_v$.  The set $J_{\phi,v}$ consists of all
$x\in\mathbb{C}_v$ for which $\{|\phi_Q(x)|_v\}_{Q\in A}$ is bounded.
It is immediate to see that $x\in J_{\phi,v}$ if and only if
$\{|\phi_{t^n}(x)|_v\}_{n\ge 1}$ is bounded.

One final note on absolute values: as noted above, the place $v\in M_K$ extends to a unique absolute value $|\cdot |_v$ on all of $\mathbb{C}_v$.  We fix
an embedding of $i: \bK  \lra \mathbb{C}_v$.  For $x \in\bK$, we denote
$| i(x) |_v$ simply as $|x|_v$, by abuse of notation.  
 
\subsection{The coefficients of $\Drin_t$}
Each Drinfeld module is isomorphic to a Drinfeld module for which all
the coefficients of $\Drin_t$ are integral at all the places in $M_K$
which do not lie over $v_{\infty}$. Indeed, we let $B\in
\Fq[t]$ be a product of all (the finitely many)
irreducible polynomials $P\in\Fq[t]$ with the property
that there exists a place $v\in M_K$ which lies over the place $(P)\in M_{\Fq(t)}$, and there exists a coefficient of $\phi_t$
which is not integral at $v$. Let $\gamma$ be a sufficiently large
power of $B$. Then $\psi:A\rightarrow K\{\tau\}$ defined by
$\psi_Q:=\gamma^{-1}\Drin_Q\gamma$ (for each $Q\in A$) is a Drinfeld
module isomorphic to $\Drin$, and all the coefficients of $\psi_t$ are
integral away from the places lying above $v_{\infty}$. Hence, from
now on, we assume that all the coefficients of $\Drin_t$ are integral
away from the places lying over $v_{\infty}$. It follows that for
every $Q\in A$, all coefficients of $\phi_Q$ are integral away from
the places lying over $v_{\infty}$.

\subsection{Integrality and reduction}
\begin{definition}
\label{S-integral definition}
For a finite set of places $S \subset M_K$ and $\alpha\in\bK$, we say
that $\beta \in\bK$ is $S$-integral with respect to $\alpha$ if for
every place $v\notin S$, and for every morphisms $\sigma,\tau:\bK\rightarrow\bK$ (which restrict to the identity on $K$) the following
are true:
\begin{itemize}
\item if $|\alpha^{\tau}|_v\le 1$, then $|\alpha^{\tau}-\beta^{\sigma}|_v\ge 1$.
\item if $|\alpha^{\tau}|_v >1$, then $|\beta^{\sigma}|_v\le 1$.
\end{itemize}
\end{definition}
We note that if $\beta$ is $S$-integral with respect to $\alpha$, then it is also $S'$-integral with respect to $\alpha$, where $S'$ is a finite set of places containing $S$. Moreover, the fact that $\beta$ is $S$-integral with respect to $\alpha$, is preserved if we replace $K$ by a finite extension. Therefore, in our results we will always assume $\alpha,\beta\in K$. 
For more details about the definition of $S$-integrality, we refer the reader to \cite{BIR}.

\begin{definition}
  The Drinfeld module $\phi$ has good reduction at a place $v$ if for
  each nonzero $a\in A$, all coefficients of $\phi_a$ are $v$-adic
  integers and the leading coefficient of $\phi_a$ is a $v$-adic unit.
  If $\phi$ does not have good reduction at $v$, then we say that
  $\phi$ has bad reduction at $v$.
\end{definition}

It is immediate to see that $\phi$ has good reduction at $v$ if and only if all coefficients of $\phi_t$ are $v$-adic integers, while the leading coefficient of $\phi_t$ is a $v$-adic unit.

We can now state our Siegel type result for Drinfeld modules.
\begin{theorem}
\label{Siegel}
With the above notation, assume in addition $K$ has only one infinite place. Let $\Gamma$ be a finitely generated $\phi$-submodule of $\mathbb{G}_a(K)$, let $\alpha\in K$, and let $S$ be a finite set of places in $M_K$. Then there are finitely many $\gamma\in\Gamma$ such that $\gamma$ is $S$-integral with respect to $\alpha$.
\end{theorem}

As mentioned in Section~\ref{intro}, we proved in \cite{findrin} that
Theorem~\ref{Siegel} holds when $\Gamma$ is a cyclic $\phi$-module
generated by a nontorsion point $\beta\in K$ and
$\alpha\in\phi_{\tor}(K)$ (see Theorem $1.1$ and Proposition $5.6$ of
\cite{findrin}). Moreover, in \cite{findrin} we did not have in our
results the extra hypothesis from Theorem~\ref{Siegel} that there
exists only one infinite place in $M_K$. Even though we believe
Theorem~\ref{Siegel} is true without this hypothesis, our method for
proving Theorem~\ref{Siegel} requires this technical hypothesis. On
the other hand, we are able to prove the following analog for Drinfeld
modules of a theorem of Silverman (see \cite{SilSiegel}) for
nonconstant morphisms of $\bP^1$ of degree greater than $1$ over a
number field, without the hypothesis of having only one infinite place
in $M_K$.

\begin{theorem}
\label{Silverman}
With the above notation, let $\beta\in K$ be a nontorsion point, and let $\alpha\in K$ be an arbitrary point. Then there are finitely many $Q\in A$ such that $\phi_Q(\beta)$ is $S$-integral for $\alpha$.
\end{theorem}
As explained before, in \cite{findrin} we proved Theorem~\ref{Silverman} in the case $\alpha$ is a torsion point in $K$.

\section{Proofs of our main results}\label{proof}

We continue with the notation from Section~\ref{notation}.
In our argument, we will be using the following key fact.
\begin{fact}
\label{Bosser}
Assume $\infty\in M_K$ is an infinite place. Let $\gamma_1,\dots,\gamma_r,\alpha\in K$. Then there exist
(negative) constants $C_0$ and $C_1$ (depending only on $\phi$,
$\gamma_1,\dots,\gamma_r, \alpha$) such that for any polynomials $P_1, \dots, P_r
\in A$ (not all constants), either
$\phi_{P_1}(\gamma_1)+\dots+\phi_{P_r}(\gamma_r)=\alpha$ or
$$
\log | \phi_{P_1} (\gamma_1) + \dots + \phi_{P_r} (\gamma_r) -
\alpha|_{\infty} \geq C_0 + C_1 \max_{1\le i\le r}(\deg(P_i)\log \deg(P_i)).
  $$
\end{fact}
Fact~\ref{Bosser} follows easily from the lower bounds for linear
forms in logarithms established by Bosser (see Th\'{e}or\`{e}me $1.1$
in \cite{Bosser}). Essentially, it is the same proof as our proof of
Proposition $3.7$ of \cite{findrin} (see in particular the derivation
of the inequality $(3.7.2)$ in \cite{findrin}). For the sake of
completeness, we will provide below a sketch of a proof of
Fact~\ref{Bosser}.

\begin{proof}[Proof of Fact~\ref{Bosser}.]
We denote by $\exp_{\infty}$ the exponential map associated to the place $\infty$ (see \cite{Goss}). We also let $\cL$ be the corresponding lattice for $\exp_{\infty}$, i.e. $\cL:=\ker(\exp_{\infty})$. Finally, let $\omega_1,\dots,\omega_d$ be an $A$-basis for $\cL$ of ``successive minima'' (see Lemma $(4.2)$ of \cite{Taguchi}). This means that for every $Q_1,\dots,Q_d\in A$, we have
\begin{equation}
\label{successive minima}
|Q_1\omega_1+\dots+Q_d\omega_d|_{\infty}=\max_{i=1}^d |Q_i\omega_i|_{\infty}.
\end{equation}

Let $u_0\in\CC$ such that $\exp_{\infty}(u_0)=\alpha$. We also let $u_1,\dots,u_r\in\CC$ such that for each $i$, we have $\exp_{\infty}(u_i)=\gamma_i$. We will find constants $C_0$ and $C_1$ satisfying the inequality from Fact~\ref{Bosser}, which depend only on $\phi$ and $u_0,u_1,\dots,u_r$.

There exists a positive constant $C_2$ such that $\exp_{\infty}$ induces an isomorphism from the ball $B:=\{z\in \bC_{\infty}\text{ : }|z|_{\infty}<C_2\}$ to itself (see Lemma $3.6$ of \cite{findrin}). If we assume there exist no constants $C_0$ and $C_1$ as in the conclusion of Fact~\ref{Bosser}, then there exist polynomials $P_1,\dots,P_r$, not all constants, such that
\begin{equation}
\label{Bosser nonzero}
\sum_{i=1}^r \phi_{P_i}(\gamma_i)\ne \alpha
\end{equation}
and $|\sum_{i=1}^r\phi_{P_i}(\gamma_i)-\alpha|_{\infty}<C_2$. Thus we can find $y\in B$ such that $|y|_{\infty}=|\sum_{i=1}^r\phi_{P_i}(\gamma_i)-\alpha|_{\infty}$ and 
\begin{equation}
\label{ecuatie pentru y}
\exp_{\infty}(y)=\sum_{i=1}^r \phi_{P_i}(\gamma_i) - \alpha.
\end{equation}
Moreover, because $\exp_{\infty}$ is an isomorphism on the metric space $B$, then for every $y'\in\CC$ such that $\exp_{\infty}(y')=\sum_{i=1}^r \phi_{P_i}(\gamma_i) - \alpha$, we have $|y'|_{\infty}\ge |y|_{\infty}$. But we know that
\begin{equation}
\label{ecuatie pentru y din nou}
\exp_{\infty}\left(\sum_{i=1}^r P_iu_i -u_0\right)=\sum_{i=1}^r \phi_{P_i}(\gamma_i) - \alpha.
\end{equation}
Therefore $|\sum_{i=1}^r P_iu_i -u_0|_{\infty} \ge |y|_{\infty}$. On the other hand, using \eqref{ecuatie pentru y} and \eqref{ecuatie pentru y din nou}, we conclude that there exist polynomials $Q_1,\dots,Q_d$ such that
$$\sum_{i=1}^r P_iu_i - u_0 = y+\sum_{i=1}^d Q_i\omega_i.$$
Hence $|\sum_{i=1}^d Q_i\omega_i|_{\infty}\le |\sum_{i=1}^r P_iu_i-u_0|_{\infty}$. Using \eqref{successive minima}, we obtain
\begin{equation}
\label{small degrees for the Q's}
\begin{split}
 \left|\sum_{i=1}^d Q_i\omega_i\right|_{\infty} = \max_{i=1}^d
 |Q_i\omega_i|_{\infty} 
 & \le \left|\sum_{i=1}^r P_iu_i-u_0\right|_{\infty}\\
 & \le \max\left(|u_0|_{\infty},\max_{i=1}^r |P_iu_i|_{\infty}\right)\\
 & \le C_3\cdot\max_{i=1}^r |P_i|_{\infty},
 \end{split}
\end{equation}
where $C_3$ is a constant depending only on $u_0,u_1,\dots,u_r$. We take logarithms of both sides in \eqref{small degrees for the Q's} and obtain
\begin{equation}
\label{small degrees for the Q's again}
\begin{split}
\max_{i=1}^d \deg Q_i & \le \max_{i=1}^r \deg P_i +\log C_3 -\min_{i=1}^d \log |\omega_i|_{\infty}\\
       & \le \max_{i=1}^r \deg P_i +C_4,
\end{split}
\end{equation}
where $C_4$ depends only on $\phi$ and $u_0,u_1,\dots,u_r$ (the dependence on the $\omega_i$ is actually a dependence on $\phi$, because the $\omega_i$ are a fixed basis of ``successive minima'' for $\phi$ at $\infty$). Using \eqref{small degrees for the Q's again} and Proposition $3.2$ of \cite{findrin} (which is a translation of the bounds for linear forms in logarithms for Drinfeld modules established in \cite{Bosser}), we conclude that there exist (negative) constants $C_0$, $C_1$, $C_5$ and $C_6$ (depending only on $\phi$, $\gamma_1,\dots,\gamma_r$ and $\alpha$) such that
\begin{equation}
\label{eticheta}
\begin{split}
  \log \left|\sum_{i=1}^r \phi_{P_i}(\gamma_i)-\alpha\right|_{\infty}
  & = \log |y|_{\infty}\\
  & =\log \left|\sum_{i=1}^r P_iu_i-u_0-\sum_{i=1}^d Q_i\omega_i\right|_{\infty}\\
  & \ge C_5 + C_6\left(\max_{i=1}^r\deg P_i +C_4\right)\log\max_{i=1}^r\left(\deg P_i+C_4\right)\\
  & \ge C_0 + C_1\left(\max_{i=1}^r\deg P_i \right)\log\max_{i=1}^r\left(\deg
    P_i\right),
\end{split}
\end{equation}
as desired.
\end{proof}

In our proofs for Theorems~\ref{Silverman} and \ref{Siegel} we will
also use the following statement, which is believed to be true, based
on communication with V.~Bosser.  Therefore we assume its validity
without proof.
\begin{statement}
\label{Bosser-wannabe}
Assume $v$ does not lie above $v_{\infty}$. Let $\gamma_1,\dots,\gamma_r,\alpha\in K$. Then there exist positive
constants $C_1,C_2,C_3$ (depending only on $v$, $\phi$,
$\gamma_1,\dots,\gamma_r$ and $\alpha$) such that for any $P_1, \dots,
P_r \in \Fq[t]$, either
$\phi_{P_1}(\gamma_1)+\dots+\phi_{P_r}(\gamma_r)=\alpha$ or
$$
\log | \phi_{P_1} (\gamma_1) + \dots + \phi_{P_r} (\gamma_r) -
\alpha|_{v} \geq - C_1 - C_2 \max_{1\le i\le r}(\deg(P_i))^{C_3}.$$
\end{statement}

Statement~\ref{Bosser-wannabe} follows after one establishes a lower
bound for linear forms in logarithms at finite places $v$. In a
private communication, V. Bosser told us that it is clear to him that
his proof (\cite{Bosser}) can be adapted to work also at finite
places with minor modifications.

We sketch here how Statement~\ref{Bosser-wannabe} would follow from a
lower bound for linear forms in logarithms at finite places. Let $v$
be a finite place and let $\exp_v$ be the formal exponential map
associated to $v$. The existence of $\exp_v$ and its convergence on a
sufficiently small ball $B_v:=\{x\in\C_v\text{ : }|x|_v < C_v\}$ is
proved along the same lines as the existence and the convergence of
the usual exponential map at infinite places for $\phi$ (see Section
$4.6$ of \cite{Goss}). In addition,
\begin{equation}
\label{expo-finite}
|\exp_v(x)|_v=|x|_v
\end{equation}
for every $x\in B_v$.
Moreover, at the expense of replacing $C_v$ with a smaller positive constant, we may assume that for each $F\in
A$, and for each $x\in B_v$, we have (see Lemma $4.2$ in \cite{findrin})
\begin{equation}
\label{trecut}
|\phi_F(x)|_v=|Fx|_v.
\end{equation} 
Assume we know the existence of the
following lower bound for (nonzero) linear forms in logarithms at a
finite place $v$.
\begin{statement}
\label{logs}
Let $u_1,\dots,u_r\in B_v$ such that for each $i$,
$\exp_v(u_i)\in\bK$. Then there exist positive constants $C_4$, $C_5$,
and $C_6$ (depending on $u_1,\dots,u_r$) such that for every
$F_1,\dots,F_r\in A$, either $\sum_{i=1}^r F_iu_i=0$, or 
$$\log \left|\sum_{i=1}^r F_i u_i \right|_v\ge - C_4 -
C_5 \left( \max_{i=1}^r \deg F_i \right)^{C_6}.$$
\end{statement}
As mentioned before, Bosser proved Statement~\ref{logs} in the case
$v$ is an infinite place (in his result, $C_6=1+\epsilon$ and
$C_4=C_{\epsilon}$ for every $\epsilon>0$).

We will now derive Statement~\ref{Bosser-wannabe} assuming Statement~\ref{logs}
holds. 

\begin{proof} ({\it That Statement~\ref{logs} implies
    Statement~\ref{Bosser-wannabe}}.)  Clearly, it suffices to prove
    Statement~\ref{Bosser-wannabe} in the case $\alpha=0$. So, let
    $\gamma_1,\dots,\gamma_r\in K$, and assume by contradiction that
    there exists an infinite sequence
    $\{F_{n,i}\}_{\substack{n\in\mathbb{N}^{*}\\ 1\le i\le r}}$ such
    that for each $n$, we have
\begin{equation}
\label{999}
-\infty < \log\left|\sum_{i=1}^r \phi_{F_{n,i}}(\gamma_i)\right|_v<  \log C_v .
\end{equation}
For each $n\ge 1$, we let $\cF_n:=\left(F_{n,1},\dots,F_{n,r}\right)\in A^r$. We view $A^r$ as an $r$-dimensional $A$-lattice inside the $r$-dimensional $\Frac(A)$-vector space $\Frac(A)^r$. In addition, we may assume that for $n\ne m$, we have $\cF_n\ne\cF_m$.

Using basic linear algebra, because the sequence
$\{F_{n,i}\}_{\substack{n\in\mathbb{N}^{*}\\ 1\le i\le r}}$ is
infinite, we can find $n_0\ge 1$ such that for every $n>n_0$, there
exist $H_n,G_{n,1},\dots,G_{n,n_0}\in A$ (not all equal to $0$) such
that
\begin{equation}
\label{1000}
H_n\cdot\cF_n = \sum_{j=1}^{n_0} G_{n,j}\cdot\cF_j.
\end{equation}
Essentially, \eqref{1000} says that $\cF_1,\dots,\cF_{n_0}$ span the
linear subspace of $\Frac(A)^r$ generated by all $\cF_n$.  Moreover,
we can choose the $H_n$ in \eqref{1000} in such a way that $\deg H_n$
is bounded independently of $n$ (e.g. by a suitable determinant of
some linearly independent subset of the first $n_0$ of the $\cF_j$).
Furthemore, there exists a constant $C_7$ such that for all $n>n_0$,
we have
\begin{equation}
\label{1000.5}
\max_{j=1}^{n_0}\deg G_{n,j}< C_7 + \max_{i=1}^r\deg F_{n,i}.
\end{equation}
Because $\left|\sum_{i=1}^r \phi_{F_{n,i}}(\gamma_i)\right|_v <C_v$,
equation \eqref{trecut} yields
\begin{equation}
\label{1001}
\left|\phi_{H_n}\left(\sum_{i=1}^r \phi_{F_{n,i}}(\gamma_i)\right)\right|_v=|H_n|_v\cdot \left|\sum_{i=1}^r \phi_{F_{n,i}}(\gamma_i)\right|_v.
\end{equation}
Using \eqref{1000}, \eqref{1001}, and the fact that $|H_n|_v\le 1$, we obtain
\begin{equation}
\label{1002}
\begin{split}
\left|\sum_{i=1}^r \phi_{F_{n,i}}(\gamma_i)\right|_v
& \ge \left|\phi_{H_n}\left(\sum_{i=1}^r \phi_{F_{n,i}}(\gamma_i)\right)\right|_v\\
& = \left|\sum_{j=1}^{n_0}\phi_{G_{n,j}}\left(\sum_{i=1}^r \phi_{F_{j,i}}(\gamma_i)\right)\right|_v.
\end{split}
\end{equation}
Since $\left|\sum_{i=1}^r\phi_{F_{j,i}}(\gamma_i)\right|_v<C_v$ for
all $1\le j\le n_0$, there exist $u_1,\dots,u_{n_0}\in B_v$ such that
for every $1\le j\le n_0$, we have
$$\exp_v(u_j)=\sum_{i=1}^r \phi_{F_{j,i}}(\gamma_i).$$
Then
Statement~\ref{logs} implies that there exist constants $C_4,C_5,C_6,
C_8, C_9$ (depending on $u_1,\dots,u_{n_0}$), such that
\begin{equation}\label{1003}
\begin{split}
  \log\left|\sum_{j=1}^{n_0}\phi_{G_{n,j}} \left(
      \sum_{i=1}^r\phi_{F_{j,i}}(\gamma_i) \right)\right|_v 
  & = \log\left|\sum_{j=1}^{n_0}G_{n,j} u_j\right|_v\\
  & \ge - C_4 - C_5 \left(\max_{j=1}^{n_0}\deg G_{n,j}\right)^{C_6}\\
  & \ge - C_8 - C_9 \left(\max_{i=1}^r\deg F_{n,i}\right)^{C_6},
\end{split}
\end{equation}
where in the first equality we used \eqref{expo-finite}, while in the last inequality we used \eqref{1000.5}. Equations
\eqref{1002} and \eqref{1003} show that
Statement~\ref{Bosser-wannabe} follows from Statement~\ref{logs}, as
desired.  
\end{proof}

Next we prove Theorem~\ref{Silverman} which will be a \emph{warm-up}
for our proof of Theorem~\ref{Siegel}. For its proof, we will only
need the following weaker (but also still conjectural) form of
Statement~\ref{Bosser-wannabe} (i.e., we only need
Statement~\ref{logs} be true for non-homogeneous $1$-forms of
logarithms).

\begin{statement}
\label{weaker Bosser-wannabe}
Assume $v$ does not lie over $v_{\infty}$. Let $\gamma,\alpha\in K$. Then there exist positive constants $C_1$, $C_2$ and $C_3$ (depending only on $v$, $\phi$, $\gamma$ and $\alpha$) such that for each polynomial $P\in\Fq[t]$, either $\phi_P(\gamma)=\alpha$ or
$$\log \left| \phi_P(\gamma)-\alpha \right|_v \ge -C_1 -C_2 \deg(P)^{C_3}.$$
\end{statement}

\begin{proof}[Proof of Theorem~\ref{Silverman}.]
The following Lemma is the key to our proof.
\begin{lemma}
\label{alt_loc_height}
For each $v\in M_K$, we have $\hhat_v(\beta) = \lim_{\deg Q\to\infty} \frac{\log | \phi_Q(\beta) - \alpha |_v}{q^{d\deg Q}}$.
\end{lemma}

\begin{proof}[Proof of Lemma~\ref{alt_loc_height}.]
Let $v\in M_K$. If $\beta\notin J_{\phi,v}$, then $|\phi_Q(\beta)|_v\to\infty$, as $\deg Q\to\infty$. Hence, if $\deg Q$ is sufficiently large, then $|\phi_Q(\beta)-\alpha|_v=|\phi_Q(\beta)|_v=\max\{ |\phi_Q(\beta)|_v,1\}$, which yields the conclusion of Lemma~\ref{alt_loc_height}.

Thus, from now on, we assume $\beta\in J_{\phi,v}$. Hence $\hhat_v(\beta)=0$, and we need to show that
\begin{equation}
\label{alt_loc_height_20}
\lim_{\deg Q\to\infty}\frac{\log |\phi_Q(\beta)-\alpha|_v}{q^{d\deg Q}}=0.
\end{equation}
Also note that since $\beta\in J_{\phi,v}$, then
$|\phi_Q(\beta)-\alpha|_v$ is bounded, and so, $\limsup_{\deg
  Q\to\infty}\frac{\log |\phi_Q(\beta)-\alpha|_v}{q^{d\deg Q}}\le 0$.
Thus, in order to prove \eqref{alt_loc_height_20}, it suffices to show
that
\begin{equation}
\label{alt_loc_height_2}
\liminf_{\deg Q\to\infty}\frac{\log |\phi_Q(\beta)-\alpha|_v}{q^{d\deg Q}}\ge 0.
\end{equation}

If v is an infinite place, then Fact~\ref{Bosser} implies that for
every polynomial $Q$ such that $\phi_Q(\beta)\ne\alpha$, we have $\log
| \phi_Q(\beta) - \alpha |_{\infty} \ge C_0 + C_1 \deg(Q)\log\deg(Q)$
(for some constants $C_0,C_1<0$).  Then taking the limit as $\deg
Q\to\infty$, we obtain \eqref{alt_loc_height_2}, as desired.

Similarly, if $v$ is a finite place, then \eqref{alt_loc_height_2} follows from Statement~\ref{weaker Bosser-wannabe}.
\end{proof}

Theorem~\ref{Silverman} follows easily using the result of Lemma~\ref{alt_loc_height}. We assume there exist infinitely many polynomials $Q_n$ such that $\phi_{Q_n}(\beta)$ is $S$-integral with respect to $\alpha$. We consider the sum
$$\Sigma:=\sum_{v\in M_K}\lim_{n\to\infty}\frac{\log |\phi_{Q_n}(\beta)-\alpha|_v}{q^{d\deg Q_n}}.$$
Using Lemma~\ref{alt_loc_height}, we obtain that $\Sigma=\hhat(\beta)>0$ (because $\beta\notin\phi_{\tor}$).

Let $\cT$ be a finite set of places consisting of all the places in $S$ along with all places $v\in M_K$ which satisfy at least one of the following conditions:
\begin{enumerate}
\item[1.] $|\beta|_v>1$.
\item[2.] $|\alpha|_v>1$.
\item[3.] $v$ is a place of bad reduction for $\phi$.
\end{enumerate}
Therefore by our choice for $\cT$ (see $1.$ and $3.$), for every $v\notin\cT$, we have $|\phi_{Q_n}(\beta)|_v\le 1$. Thus, using also $2.$, we have $|\phi_{Q_n}(\beta)-\alpha|_v\le 1$. On the other hand, $\phi_{Q_n}(\beta)$ is also $\cT$-integral with respect to $\alpha$. Hence, because of $2.$, then for all $v\notin\cT$, we have $|\phi_{Q_n}(\beta)-\alpha|_v\ge 1$. We conclude that for every $v\notin\cT$, and for every $n$, we have $|\phi_{Q_n}(\beta)-\alpha|_v=1$. This allows us to interchange the summation and the limit in the definition of $\Sigma$ (because then $\Sigma$ is a finite sum over all $v\in\cT$). We obtain
$$\Sigma=\lim_{n\to\infty}\frac{1}{q^{d\deg Q_n}}\sum_{v\in M_K}\log |\phi_{Q_n}(\beta)-\alpha|_v=0,$$
by the product formula applied to each $\phi_{Q_n}(\beta)-\alpha$. On the other hand, we already showed that $\Sigma=\hhat(\beta)>0$. This contradicts our assumption that there are infinitely many polynomials $Q$ such that $\phi_Q(\beta)$ is $S$-integral with respect to $\alpha$, and concludes the proof of Theorem~\ref{Silverman}.
\end{proof}

Before proceeding to the proof of Theorem~\ref{Siegel}, we prove several facts about local heights. In Lemma~\ref{jumps at an infinite place} we will use the technical assumption of having only one infinite place in $K$.

From now on, let $\phi_t=\sum_{i=0}^d a_i\tau^i$. As explained in Section~\ref{notation}, we may assume each $a_i$ is integral away from $v_{\infty}$. Also, from now on, we work under the assumption that there exists a \emph{unique} place $\infty\in M_K$ lying above $v_{\infty}$.

\begin{fact}
\label{jumps}
For every place $v$ of $K$, there exists $M_v>0$ such that for each $x\in K$, if $|x|_v>M_v$, then for every nonzero $Q\in A$, we have $|\phi_{Q}(x)|_v>M_v$. Moreover, if $|x|_v>M_v$, then $\hhat_v(x)=\log |x|_v + \frac{\log |a_d|_v}{q^d-1}>0$.
\end{fact}

Fact~\ref{jumps} is proved in Lemma $4.4$ of \cite{findrin}. In particular, Fact~\ref{jumps} shows that for each $v\in M_K$ and for each $x\in K$, we have $\hhat_v(x)\in\mathbb{Q}$. Indeed, for every $x\in K$ of positive local canonical height at $v$, there exists a polynomial $P$ such that $|\phi_P(x)|_v>M_v$. Then $\hhat_v(x)=\frac{\hhat_v(\phi_P(x))}{q^{d\deg P}}$. By Fact~\ref{jumps}, we already know that $\hhat_v(\phi_P(x))\in\mathbb{Q}$. Thus also $\hhat_v(x)\in\mathbb{Q}$.

\begin{fact}
\label{jumps-2}
Let $v\in M_K\setminus\{\infty\}$. There exists a positive constant $N_v$, and there exists a nonzero polynomial $Q_v$, such that for each $x\in K$, the following statements are true
\begin{enumerate}
\item if $|x|_v\le N_v$, then for each $Q\in A$, we have $|\phi_Q(x)|_v\le |x|_v\le N_v$.
\item either $|\phi_{Q_v}(x)|_v\le N_v$, or $|\phi_{Q_v}(x)|_v>M_v$.
\end{enumerate}
\end{fact}

\begin{proof}[Proof of Fact~\ref{jumps-2}.]
This was proved in \cite{locleh}. It is easy to see that 
$$N_v:=\min_{1\le i\le d}|a_i|_v^{-\frac{1}{q^i-1}}$$
satisfies condition $(i)$, but the proof of $(ii)$ is much more
complicated. In \cite{locleh}, the first author proved that there
exists a positive integer $d_v$ such that for every $x\in K$, there
exists a polynomial $Q$ of degree at most $d_v$ such that either
$|\phi_Q(x)|_v>M_v$, or $|\phi_Q(x)|_v\le N_v$ (see Remark $5.12$ which
is valid for every place which does not lie over $v_{\infty}$). Using
Fact~\ref{jumps} and $(i)$, we conclude that the polynomial
$Q_v:=\prod_{\deg P\le d_v} P$ satisfies property $(ii)$.
\end{proof}

Using Facts~\ref{jumps} and \ref{jumps-2} we prove the following important result valid for finite places.
\begin{lemma}
\label{jumps at a finite place}
Let $v\in M_K\setminus\{\infty\}$. Then there exists a positive integer $D_v$ such that for every $x\in K$, we have $D_v\cdot\hhat_v(x)\in\mathbb{N}$. If in addition we assume $v$ is a place of good reduction for $\phi$, then we may take $D_v=1$.
\end{lemma}

\begin{proof}[Proof of Lemma~\ref{jumps at a finite place}.]
Let $x\in K$. If $\hhat_v(x)=0$, then we have nothing to show. Therefore, assume from now on that $\hhat_v(x)>0$. Using $(ii)$ of Fact~\ref{jumps-2}, there exists a polynomial $Q_v$ (depending only on $v$, and not on $x$) such that $|\phi_{Q_v}(x)|_v>M_v$ (clearly, the other option from $(ii)$ of Lemma~\ref{jumps-2} is not available because we assumed that $\hhat_v(x)>0$). Moreover, using the definition of the local height, and also Fact~\ref{jumps}, 
\begin{equation}
\label{integer}
\hhat_v(x)=\frac{\hhat_v(\phi_{Q_v}(x))}{q^{d\deg Q_v}}=\frac{\log |\phi_{Q_v}(x)|_v + \frac{\log |a_d|_v}{q^d-1}}{q^{d\deg Q_v}}.
\end{equation}
Because both $\log |\phi_{Q_v}(x)|_v$ and $\log |a_d|_v$ are integer numbers, \eqref{integer} yields the conclusion of Lemma~\ref{jumps at a finite place} (we may take $D_v=q^{d\deg Q_v}(q^d-1)$).

The second part of Lemma~\ref{jumps at a finite place} follows immediately from Lemma $4.13$ of \cite{locleh}. Indeed, if $v$ is a place of good reduction for $\phi$, then $|x|_v>1$ (because we assumed $\hhat_v(x)>0$). But then, $\hhat_v(x)=\log |x|_v$ (here we use the fact that $v$ is a place of good reduction, and thus $a_d$ is a $v$-adic unit). Hence $\hhat_v(x)\in\mathbb{N}$, and we may take $D_v=1$.
\end{proof}

The following result is an immediate corollary of Fact~\ref{jumps at a finite place}.
\begin{cor}
\label{uniform-D}
There exists a positive integer $D$ such that for every $v\in M_K\setminus\{\infty\}$, and for every $x\in K$, we have $D\cdot\hhat_v(x)\in\mathbb{N}$.
\end{cor}

Next we prove a similar result as in Lemma~\ref{jumps at a finite place} which is valid for the \emph{only} infinite place of $K$.
\begin{lemma}
\label{jumps at an infinite place}
There exists a positive integer $D_{\infty}$ such that for every $x\in K$, either $\hhat_v(x)>0$ for some $v\in M_K\setminus\{\infty\}$, or $D_{\infty}\cdot\hhat_{\infty}(x)\in\mathbb{N}$.
\end{lemma}

Before proceeding to its proof, we observe that we cannot remove the
assumption that $\hhat_v(x)=0$ for every finite place $v$, in order to
obtain the existence of $D_{\infty}$ in the statement of
Lemma~\ref{jumps at an infinite place}. Indeed, we know that in $K$
there are points of arbitrarily small (but positive) local height at
$\infty$ (see Example $6.1$ from \cite{locleh}). Therefore, there
exists \emph{no} positive integer $D_{\infty}$ which would clear all
the possible denominators for the local heights at $\infty$ of those
points. However, it turns out (as we will show in the proof of
Lemma~\ref{jumps at an infinite place}) that for such points $x$ of
\emph{very} small local height at $\infty$, there exists some other
place $v$ for which $\hhat_v(x)>0$.

\begin{proof}[Proof of Lemma~\ref{jumps at an infinite place}.]
  Let $x\in K$. If $x\in\phi_{\tor}$, then we have nothing to prove
  (every positive integer $D_{\infty}$ would work because
  $\hhat_{\infty}(x)=0$). Thus, we assume $x$ is a nontorsion point.
  If $\hhat_v(x)>0$ for some place $v$ which does not lie over
  $v_{\infty}$, then again we are done. So, assume from now on that
  $\hhat_v(x)=0$ for every finite place $v$.

  By proceeding as in the proof of Lemma~\ref{jumps at a finite
    place}, it suffices to show that there exists a polynomial
  $Q_{\infty}$ of degree bounded independently of $x$ such that
  $|\phi_{Q_{\infty}}(x)|_{\infty}>M_{\infty}$ (with the notation as
  in Fact~\ref{jumps}). This is proved in Theorem $4.4$ of \cite{mw2}.
  The first author showed in \cite{mw2} that there exists a positive
  integer $L$ (depending only on the number of places of bad reduction
  of $\phi$) such that for every nontorsion point $x$, there exists a
  place $v\in M_K$, and there exists a polynomial $Q$ of degree less
  than $L$ such that $|\phi_Q(x)|_v>M_v$. Because we assumed that
  $\hhat_v(x)=0$ for every $v\ne\infty$, then the above statement
  yields the existence of $D_{\infty}$.
\end{proof}

We will prove Theorem \ref{Siegel} by showing that a certain $\limsup$ is
positive.  This will contradict the existence of infinitely many
$S$-integral points in a finitely generated $\phi$-submodule.
Our first step will be a result about the $\liminf$ of the sequences which will appear in the proof of Theorem~\ref{Siegel}.  

\begin{lemma}\label{inf}
  Suppose that $\Gamma$ is a torsion-free $\phi$-submodule of $\bG_a(K)$ generated by
  elements $\gamma_1, \dots, \gamma_r$. For each $i\in\{1,\dots,r\}$ let $(P_{n,i})_{n\in\mathbb{N}^{*}}\subset\Fq[t]$ be a sequence of polynomials such that for each $m\ne n$, the $r$-tuples $(P_{n,i})_{1\le i\le r}$ and $(P_{m,i})_{1\le i\le r}$ are distinct. Then for every $v\in M_K$, we have
\begin{equation}
\label{non-negative}
\liminf_{n\to\infty} \frac{ \log |\sum_{i=1}^r
  \phi_{P_{n,i}}(\gamma_i)-\alpha |_v}{\sum_{i=1}^r q^{d\deg P_{n,i}}}
\ge 0.
\end{equation}
\end{lemma}
\begin{proof}
  Suppose that for some $\epsilon>0$, there exists a sequence $(n_k)_{k\ge
    1}\subset\mathbb{N}^{*}$ such that $\sum_{i=1}^r \phi_{P_{n_k,i}}(\gamma_i)\ne\alpha$ and
\begin{equation}
\label{still non-negative}
\frac{ \log |\sum_{i=1}^r \phi_{P_{n_k,i}}(\gamma_i)-\alpha |_v}{\sum_{i=1}^r q^{d\deg P_{n_k,i}}} < -\epsilon,
\end{equation}
for every $k\ge 1$.  But taking the lower bound from Fact~\ref{Bosser}
or Statement~\ref{Bosser-wannabe} (depending on whether $v$ is the
infinite place or not) and dividing through by $\sum_{i=1}^r q^{d\deg
  P_{n_k,i}}$, we see that this is impossible.
\end{proof}

The following proposition is the key technical result required to
prove Theorem \ref{Siegel}.  This proposition plays the same role that Lemma~\ref{alt_loc_height} plays in the proof of Theorem~\ref{Silverman}, or that
Corollary $3.13$ plays in the proof of Theorem $1.1$ from \cite{findrin}.  Note that is does not
provide an exact formula for the canonical height of a point, however;
it merely shows that a certain limit is positive.  This will suffice for
our purposes since we only need that a certain sum of limits be
positive in order to prove Theorem \ref{Siegel}.

\begin{proposition}\label{sup}
 Let $\Gamma$ be a torsion-free $\phi$-submodule of $\bG_a(K)$
  generated by elements $\gamma_1, \dots, \gamma_r$. For each $i\in\{1,\dots,r\}$ let $(P_{n,i})_{n\in\mathbb{N}^{*}}\subset\Fq[t]$ be a sequence of polynomials such that for each $m\ne n$, the $r$-tuples $(P_{n,i})_{1\le i\le r}$ and $(P_{m,i})_{1\le i\le r}$ are distinct.  Then there exists a place $v\in M_K$
  such that
\begin{equation}
\label{positive}
\limsup_{n\to\infty} \frac{ \log |\sum_{i=1}^r
\phi_{P_{n,i}}(\gamma_i) - \alpha |_v}{\sum_{i=1}^r q^{d\deg
P_{n,i}}}> 0.
\end{equation}
\end{proposition}

\begin{proof}
Using the triangle inequality for the $v$-adic norm, and the fact that
$$\lim_{n\to\infty}\sum_{i=1}^r q^{d\deg P_{n,i}} = +\infty,$$
we
conclude that proving that \eqref{positive} holds is equivalent to
proving that for some place $v$, we have
\begin{equation}
\label{positive-2}
\limsup_{n\to\infty} \frac{ \log |\sum_{i=1}^r \phi_{P_{n,i}}(\gamma_i) |_v}{\sum_{i=1}^r q^{d\deg P_{n,i}}} > 0.
\end{equation}
We also observe that it suffices to prove Proposition~\ref{sup} for a subsequence $(n_k)_{k\ge 1}\subset\mathbb{N}^{*}$.

We prove \eqref{positive-2} by induction on $r$. If $r=1$, then by
Corollary $3.13$ of \cite{findrin} (see also our Lemma~\ref{alt_loc_height}),
\begin{equation}
\label{cazul-1}
\limsup_{\deg P\to\infty}\frac{\log |\phi_P(\gamma_1)|_v} {q^{d\deg
    P}} = \hhat_v(\gamma_1)
\end{equation}
and because $\gamma_1\notin\phi_{\tor}$,
there exists a place $v$ such that $\hhat_v(\gamma_1) > 0$, thus
proving \eqref{positive-2} for $r=1$. Therefore, we assume
\eqref{positive-2} is established for all $\phi$-submodules $\Gamma$ of rank less
than $r$ and we will prove it for $\phi$-submodules of rank $r$.

In the course of our argument for proving \eqref{positive-2},
we will replace several times a given sequence with a subsequence of
itself (note that passing to a subsequence can only make the $\limsup$
smaller). For the sake of not clustering the notation, we will drop
the extra indices which would be introduced by dealing with the
subsequence.

Let $S_0$ be the set of places $v\in M_K$ for which there exists some $\gamma\in\Gamma$ such that $\hhat_v(\gamma)>0$. The following easy fact will be used later in our argument.
\begin{fact}
\label{S_0 is finite}
The set $S_0$ is finite.
\end{fact}

\begin{proof}[Proof of Fact~\ref{S_0 is finite}.]
We claim that $S_0$ equals the \emph{finite} set $S_0'$ of places $v\in M_K$ for which there exists $i\in\{1,\dots,r\}$ such that $\hhat_v(\gamma_i)>0$. Indeed, let $v\in M_K\setminus S_0'$. Then for each $i\in\{1,\dots,r\}$ we have $\hhat_v(\gamma_i)=0$. Moreover, for each $i\in\{1,\dots,r\}$ and for each $Q_i\in\Fq[t]$, we have 
\begin{equation}
\label{Q_i times 0}
\hhat_v(\phi_{Q_i}(\gamma_i))=\deg(\phi_{Q_i})\cdot\hhat_v(\gamma_i)=0.
\end{equation}
Using \eqref{Q_i times 0} and the triangle inequality for the local canonical height, we obtain that $$\hhat_v\left(\sum_{i=1}^r\phi_{Q_i}(\gamma_i)\right)=0.$$
This shows that indeed $S_0=S_0'$, and concludes the proof of Fact~\ref{S_0 is finite}. 
\end{proof}

If the sequence $(n_k)_{k\ge 1}\subset\mathbb{N}^{*}$ has the property
that for some $j\in\{1,\dots,r\}$, we have
\begin{equation}
\label{elimination}
\lim_{k\to\infty}\frac{q^{d\deg P_{n_k,j}}} {\sum_{i=1}^r q^{d\deg
    P_{n_k,i}}} = 0,
\end{equation}
then the inductive hypothesis will yield the desired conclusion.
Indeed, by the induction hypothesis, and also using
\eqref{elimination}, there exists $v\in S_0$ such that
\begin{equation}
\label{induction-positive}
\limsup_{k\to\infty}\frac{\log |\sum_{i\ne
    j}\phi_{P_{n_k,i}}(\gamma_i)|_v}{\sum_{i=1}^r q^{d\deg P_{n_k,i}}}
> 0.
\end{equation}
If $\hhat_v(\gamma_j)=0$, then
$\left|\phi_{P_{n_k,j}}(\gamma_j)\right|_v$ is bounded as
$k\to\infty$. Thus, for large enough $k$,
$$\left|\sum_{i=1}^r\phi_{P_{n_k,i}}(\gamma_i)\right|_v =
\left|\sum_{i\ne j}\phi_{P_{n_k,i}}(\gamma_i)\right|_v$$
and so,
\eqref{induction-positive} shows that \eqref{positive-2} holds.

Now, if $\hhat_v(\gamma_j)>0$, then we proved in Lemma $4.4$ of
\cite{findrin} that
\begin{equation}
\label{large degree}
\log |\phi_P(\gamma_j)|_v - q^{d\deg P}\hhat_v(\gamma_j)
\end{equation}
is uniformly bounded as $\deg P\to\infty$ (note that this follows
easily from simple arguments involving geometric series and
coefficients of polynomials). Therefore, using \eqref{elimination}, we
obtain
\begin{equation}
\label{induction-0}
\lim_{k\to\infty}\frac{\log \left|\phi_{P_{n_k,j}}(\gamma_j)\right|_v}{\sum_{i=1}^r q^{d\deg P_{n_k,i}}}=0.
\end{equation}
Using \eqref{induction-positive} and \eqref{induction-0}, we conclude that for large enough $k$, 
$$\left|\sum_{i=1}^r\phi_{P_{n_k,i}}(\gamma_i)\right|_v = \left|\sum_{i\ne j} \phi_{P_{n_k,i}}(\gamma_i)\right|_v$$
and so,
\begin{equation}
\label{squeeze}
\limsup_{k\to\infty}\frac{\log \left|\sum_{i=1}^r\phi_{P_{n_k,i}}(\gamma_i)\right|_v}{\sum_{i=1}^r q^{d\deg P_{n_k,i}}}>0,
\end{equation}
as desired. Therefore, we may assume from now on that there exists
$B\ge 1$ such that for every $n$,
\begin{equation}
\label{squeezed}
\frac{\max_{1\le i\le r} q^{d\deg P_{n,i}}}{\min_{1\le i\le r} q^{d\deg P_{n,i}}}\le B\text{ or equivalently,}
\end{equation}
\begin{equation}
\label{squeezed-2}
\max_{1\le i\le r} \deg P_{n,i} - \min_{1\le i\le r} \deg P_{n,i} \le \frac{\log_q B}{d}.
\end{equation}

We will prove that \eqref{positive-2} holds for some place $v$ by doing an analysis at each place $v\in S_0$. We know that $|S_0|\ge 1$ because all $\gamma_i$ are nontorsion.

Our strategy is to show that in case \eqref{positive-2} does not hold, then we can find $\delta_1,\dots,\delta_r\in\Gamma$, and we can find a sequence $(n_k)_{k\ge 1}\subset\mathbb{N}^{*}$, and a sequence of polynomials $\left(R_{k,i}\right)_{\substack{k\in\mathbb{N}^{*}\\ 1\le i\le r}}$ such that 
\begin{equation}
\label{condition-1}
\sum_{i=1}^r \phi_{P_{n_k,i}}(\gamma_i) = \sum_{i=1}^r \phi_{R_{k,i}}(\delta_i)\text{ and}
\end{equation}
\begin{equation}
\label{condition-2}
\sum_{i=1}^r \hhat_v(\delta_i) < \sum_{i=1}^r \hhat_v(\gamma_i)\text{ and}
\end{equation}
\begin{equation}
\label{condition-squeezed}
0 < \liminf_{k\to\infty} \frac{\sum_{i=1}^r q^{d\deg P_{n_k,i}}}{\sum_{i=1}^r q^{d\deg R_{k,i}}}\le\limsup_{k\to\infty}\frac{\sum_{i=1}^r q^{d\deg P_{n_k,i}}}{\sum_{i=1}^r q^{d\deg R_{k,i}}}<+\infty.
\end{equation}
Equation \eqref{condition-1} will enable us to replace the $\gamma_i$ by the $\delta_i$ and proceed with our analysis of the latter. Inequality \eqref{condition-2} combined with Corollary~\ref{uniform-D} and Lemma~\ref{jumps at an infinite place} will show that for each such $v$, in a finite number of steps we either construct a sequence $\delta_i$ as above for which all $\hhat_v(\delta_i)=0$, \emph{or} \eqref{positive-2} holds for $\delta_1,\dots,\delta_r$ and the corresponding polynomials $R_{k,i}$, i.e.
\begin{equation}
\label{positive-2-delta}
\limsup_{k\to\infty} \frac{ \log |\sum_{i=1}^r \phi_{R_{k,i}}(\delta_i) |_v}{\sum_{i=1}^r q^{d\deg R_{k,i}}} > 0.
\end{equation}
Equation \eqref{condition-squeezed} shows that \eqref{positive-2} is equivalent to \eqref{positive-2-delta} (see also \eqref{condition-1}).

We start with $v\in S_0\setminus\{\infty\}$. As
proved in Lemma $4.4$ of \cite{findrin}, for each $i\in\{1,\dots,r\}$
such that $\hhat_v(\gamma_i)>0$, there exists a positive integer $d_i$
such that for every polynomial $Q_i$ of degree at least $d_i$, we have
\begin{equation}
\label{formula}
\log |\phi_{Q_i}(\gamma_i)|_v = q^{d\deg Q_i}\hhat_v(\gamma_i) -
\frac{\log |a_d|_v} {q^d-1}.
\end{equation}
We know that for each $i$, we have $\lim_{n\to\infty} \deg
P_{n,i}=+\infty$ because of \eqref{squeezed-2}. Hence, for each $n$
sufficiently large, and for each $i\in\{1,\dots,r\}$ such that
$\hhat_v(\gamma_i)>0$, we have
\begin{equation}
\label{formula-2}
\log |\phi_{P_{n,i}}(\gamma_i)|_v=q^{d\deg P_{n,i}}\hhat_v(\gamma_i) -
\frac{\log |a_d|_v}{q^d-1}.
\end{equation}

We now split the problem into two cases.

\noindent {\bf {\emph{Case 1.}}} There exists an infinite subsequence
$(n_k)_{k\ge 1}$ such that for every $k$, we have
\begin{equation}
\label{norm of the sum is exact}
\left|\sum_{i=1}^r \phi_{P_{n_k,i}}(\gamma_i)\right|_v=\max_{1\le i\le r} \left|\phi_{P_{n_k,i}}(\gamma_i)\right|_v.
\end{equation}
For the sake of not clustering the notation, we drop the index $k$ from \eqref{norm of the sum is exact} (note that we need to prove \eqref{positive-2} only for a \emph{subsequence}). At the expense of replacing again $\mathbb{N}^{*}$ by a subsequence, we may also assume that for some \emph{fixed} $j\in\{1,\dots,r\}$, we have
\begin{equation}
\label{norm of the sum is exact for j}
\left|\sum_{i=1}^r \phi_{P_{n,i}}(\gamma_i)\right|_v=\max_{i=1}^r \left|\phi_{P_{n,i}}(\gamma_i)\right|= \left|\phi_{P_{n,j}}(\gamma_j)\right|_v,
\end{equation}
for all $n\in\mathbb{N}^{*}$. Because we know that there exists
$i\in\{1,\dots,r\}$ such that $\hhat_v(\gamma_i)>0$, then for such
$i$, we know $|\phi_{P_{n,i}}(\gamma_i)|_v$ is unbounded (as
$n\to\infty$). Therefore, using \eqref{norm of the sum is exact for
  j}, we conclude that also $|\phi_{P_{n,j}}(\gamma_j)|_v$ is
unbounded (as $n\to\infty$). In particular, this means that
$\hhat_v(\gamma_j)>0$.

Then using \eqref{formula-2} for $\gamma_j$, we obtain that
\begin{equation}
\begin{split}
  & \limsup_{n\to\infty} \frac{\log
    |\sum_{i=1}^r\phi_{P_{n,i}}(\gamma_i)|_v}{\sum_{i=1}^r q^{d\deg
      P_{n,i}}} \\
  & = \limsup_{n\to\infty} \frac{\log
    |\phi_{P_{n,j}}(\gamma_j)|_v}{\sum_{i=1}^r q^{d\deg P_{n,i}}}\\
  & = \limsup_{n\to\infty}\frac{q^{d\deg P_{n,j}} \hhat_v(\gamma_j) -
    \frac{\log |a_d|_v}{q^d-1}} {\sum_{i=1}^r q^{d\deg P_{n,i}}}\\
  & = \lim_{n\to\infty} \frac{q^{d\deg P_{n,j}} \hhat_v(\gamma_j) -
    \frac{\log |a_d|_v} {q^d-1}}{q^{d\deg P_{n,j}}} \cdot \limsup_{n\to\infty}\frac{q^{d\deg P_{n,j}}}{\sum_{i=1}^r q^{d\deg P_{n,i}}}\\
  & > 0,
\end{split}
\end{equation}
since
$$\lim_{n\to\infty}\frac{q^{d\deg P_{n,j}}\hhat_v(\gamma_j)-\frac{\log
    |a_d|_v}{q^d-1}}{q^{d\deg P_{n,j}}} = \hhat_v(\gamma_j) > 0 \text{ and}$$
$$\limsup_{n\to\infty}\frac{q^{d\deg P_{n,j}}}{\sum_{i=1}^r q^{d\deg
    P_{n,i}}}>0 \quad \text{because of \eqref{squeezed-2}.}$$

\noindent {\bf {\emph{Case 2.}}} For all but finitely many $n$, we
have
\begin{equation}
\label{norm of the sum is not exact}
\left|\sum_{i=1}^r \phi_{P_{n,i}}(\gamma_i)\right|_v < \max_{1\le i\le r} \left|\phi_{P_{n,i}}(\gamma_i)\right|_v.
\end{equation}
Using the pigeonhole principle, there exists an infinite sequence
$(n_k)_{k\ge 1}\subset\mathbb{N}^{*}$, and there exist $j_1,\dots,
j_s\in\{1,\dots,r\}$ (where $s\ge 2$) such that for each $k$, we have
  \begin{equation}
\label{all the j's}
|\phi_{P_{n_k,j_1}}(\gamma_{j_1})|_v =\dots  =
|\phi_{P_{n_k,j_s}}(\gamma_{j_s})|_v > \max_{i\in\{1,\dots,r\}\setminus\{j_1,\dots,j_s\}} |\phi_{P_{n_k,i}}(\gamma_i)|_v.
\end{equation}
Again, as we did before, we drop the index $k$ from the above subsequence of $\mathbb{N}^{*}$. Using \eqref{all the j's} and the fact that there exists $i\in\{1,\dots,r\}$ such that $\hhat_v(\gamma_i)>0$, we conclude that for all $1\le i\le s$, we have
$\hhat_v(\gamma_{j_i})>0$. Hence, using \eqref{formula-2} in
\eqref{all the j's}, we obtain that for sufficiently large $n$, we
have
\begin{equation}
\label{relations}
q^{d\deg P_{n,j_1}}\hhat_v(\gamma_{j_1})=\dots =q^{d\deg P_{n,j_s}}\hhat_v(\gamma_{j_s}).
\end{equation}
Without loss of generality, we may assume
$\hhat_v(\gamma_{j_1})\ge\hhat_v(\gamma_{j_i})$ for all
$i\in\{2,\dots,s\}$. Then \eqref{relations} yields that $\deg
P_{n,j_i}\ge\deg P_{n,j_1}$ for $i>1$. We divide (with quotient and
remainder) each $P_{n,j_i}$ (for $i>1$) by $P_{n,j_1}$ and for each
such $j_i$, we obtain
\begin{equation}
\label{remainder-quotient}
P_{n,j_i}=P_{n,j_1}\cdot C_{n,j_i} + R_{n,j_i},
\end{equation}
where $\deg R_{n,j_i} <\deg P_{n,j_1}\le \deg P_{n,j_i}$. Using \eqref{squeezed-2}, we conclude that $\deg C_{n,j_i}$ is uniformly bounded as $n\to\infty$. This means that, at the expense of passing to another subsequence of $\mathbb{N}^{*}$, we may assume that there exist polynomials $C_{j_i}$ such that
$$C_{n,j_i}=C_{j_i}\text{ for all $n$}.$$
We let $R_{n,i}:=P_{n,i}$ for each $n$ and for each $i\in\{1,\dots,r\}\setminus\{j_2,\dots,j_s\}$.

Let $\delta_i$ for $i\in\{1,\dots,r\}$ be defined as follows:
$$\delta_i:=\gamma_i\text{ if $i\ne j_1$; and}$$
$$\delta_{j_1}:=\gamma_{j_1}+\sum_{i=2}^s\phi_{C_{j_i}}(\gamma_{j_i}).$$
Then for each $n$, using \eqref{remainder-quotient} and the definition of the $\delta_i$ and $R_{n,i}$, we obtain
\begin{equation}
\label{the same}
\sum_{i=1}^r \phi_{P_{n,i}}(\gamma_i)=\sum_{i=1}^r \phi_{R_{n,i}}(\delta_i).
\end{equation}
Using \eqref{squeezed-2} and the definition of the $R_{n,i}$ (in particular, the fact that $R_{n,j_1}=P_{n,j_1}$ and $\deg R_{n,j_i}<\deg P_{n,j_1}$ for $2\le i\le s$), it is immediate to see that
\begin{equation}
\label{degrees}
0 < \liminf_{n\to\infty} \frac{\sum_{i=1}^r q^{d\deg P_{n,i}}}{\sum_{i=1}^r q^{d\deg R_{n,i}}}\le\limsup_{n\to\infty}\frac{\sum_{i=1}^r q^{d\deg P_{n,i}}}{\sum_{i=1}^r q^{d\deg R_{n,i}}}<+\infty.
\end{equation}
Moreover, because of \eqref{the same} and \eqref{degrees}, we get that
\begin{equation}
\label{positive-4}
\limsup_{n\to\infty}\frac{\log |\sum_{i=1}^r \phi_{P_{n,i}}(\gamma_i)|_v}{\sum_{i=1}^r q^{d\deg P_{n,i}}}>0
\end{equation}
if and only if
\begin{equation}
\label{positive-5}
\limsup_{n\to\infty}\frac{\log |\sum_{i=1}^r \phi_{R_{n,i}}(\delta_i)|_v}{\sum_{i=1}^r q^{d\deg R_{n,i}}}>0.
\end{equation}
We claim that if $\hhat_v(\delta_{j_1})\ge \hhat_v(\gamma_{j_1})$,
then \eqref{positive-5} holds (and so, also \eqref{positive-4} holds).
Indeed, in that case, for large enough $n$, we have
\begin{equation}
\label{large norm}
\begin{split}
  \log |\phi_{R_{n,j_1}}(\delta_{j_1})|_v  & = q^{d\deg R_{n,j_1}}\hhat_v(\delta_{j_1})-\frac{\log |a_d|_v}{q^d-1}\\
  & \ge q^{d\deg P_{n,j_1}}\hhat_v(\gamma_{j_1})-\frac{\log
    |a_d|_v}{q^d-1}\\
& =\log |\phi_{P_{n,j_1}}(\gamma_{j_1})|_v\\
&  >\max_{i=2}^s \log |\phi_{R_{n,j_i}}(\gamma_{j_i})|_v,
\end{split}
\end{equation}
where in the last inequality from \eqref{large norm} we used \eqref{relations} and \eqref{formula-2}, and that for each $i\in\{2,\dots,s\}$ we have $\deg R_{n,j_i}<\deg P_{n,j_i}$. 
Moreover, using \eqref{large norm} and \eqref{all the j's}, together with the definition of the $R_{n,i}$ and the $\delta_i$, we conclude that for large enough $n$, we have
\begin{equation}
\label{large norm 2}
\begin{split}
  \log \left|\sum_{i=1}^r \phi_{R_{n,i}}(\delta_i)\right|_v & = \log \left|\phi_{R_{n,j_1}}(\delta_{j_1})\right|_v\\
  & = q^{d\deg P_{n,j_1}}\hhat_v(\gamma_{j_1})-\frac{\log
    |a_d|_v}{q^d-1}.
\end{split}
\end{equation}
Because $R_{n,j_1}=P_{n,j_1}$, equations \eqref{squeezed-2} and \eqref{degrees} show that
\begin{equation}
\label{degrees-23}
\limsup_{n\to\infty}\frac{q^{d\deg R_{n,j_1}}}{\sum_{i=1}^r q^{d\deg R_{n,i}}} > 0.
\end{equation}
Equations \eqref{large norm 2} and \eqref{degrees-23} show that we are now in {\bf \emph{Case 1}} for the sequence $(R_{n,i})_{\substack{n\in\mathbb{N}^{*}\\1\le i\le r}}$. Hence 
\begin{equation}
   \limsup_{n\to\infty}\frac{\log |\sum_{i=1}^r\phi_{R_{n,i}}(\delta_i)|_v}{\sum_{i=1}^r q^{d\deg R_{n,i}}}>0,
\end{equation}
as desired.

Assume from now on that $\hhat_v(\delta_{j_1})<\hhat_v(\gamma_{j_1})$.
Because $v\in M_K\setminus\{\infty\}$, using Corollary~\ref{uniform-D}
and also using that if $i\ne j_1$, then $\delta_i=\gamma_i$, we
conclude
$$\sum_{i=1}^r\hhat_v(\gamma_i)-\sum_{i=1}^r\hhat_v(\delta_i)\ge\frac{1}{D}.$$
Our goal is to prove \eqref{positive-4} by proving \eqref{positive-5}.
Because we replace some of the polynomials $P_{n,i}$ with other
polynomials $R_{n,i}$, it may very well be that \eqref{squeezed-2} is
no longer satisfied for the polynomials $R_{n,i}$.  Note that in this
case, using induction and arguing as in equations \eqref{elimination}
through \eqref{squeeze}, we see that
$$\limsup_{n\to\infty}\frac{\log
  |\sum_{j=1}^r\phi_{R_{n,j}}(\delta_j)|_w}{\sum_{j=1}^r q^{d\deg
    R_{n,j}}}>0,$$
for some place $w$. This would yield that (see \eqref{the same} and \eqref{degrees})
$$\limsup_{n\to\infty}\frac{\log
  |\sum_{j=1}^r\phi_{P_{n,j}}(\gamma_j)|_w}{\sum_{j=1}^r q^{d\deg
    P_{n,j}}}>0,$$
as desired.
Hence, we may assume again that \eqref{squeezed-2} holds.

We continue the above analysis this time with the $\gamma_i$ replaced
by $\delta_i$. Either we prove \eqref{positive-5} (and so, implicitly,
\eqref{positive-4}), or we replace the $\delta_i$ by other elements in
$\Gamma$, say $\beta_i$ and we decrease even further the sum of their
local heights at $v$:
$$\sum_{i=1}^r
\hhat_v(\delta_i)-\sum_{i=1}^r\hhat_v(\beta_i)\ge\frac{1}{D}.$$
The
above process cannot go on infinitely often because the sum of the
local heights $\sum_{i=1}^r \hhat_v(\gamma_i)$ is decreased each time
by at least $\frac{1}{D}$. Our process ends when we cannot replace
anymore the eventual $\zeta_i$ by new $\beta_i$. Thus, at the final
step, we have $\zeta_1,\dots,\zeta_r$ for which we cannot further
decrease their sum of local canonical heights at $v$. This happens
either because all $\zeta_i$ have local canonical height equal to $0$,
or because we already found a sequence of polynomials $T_{n,i}$ for
which
\begin{equation}
\label{still part of the sequence}
\limsup_{n\to\infty}\frac{\log |\sum_{i=1}^r\phi_{T_{n,i}}(\zeta_i)|_v}{\sum_{i=1}^r q^{d\deg T_{n,i}}}>0.
\end{equation}
Since
\begin{equation}
\label{the same 2}
\sum_{i=1}^r \phi_{P_{n,i}}(\gamma_i)=\sum_{i=1}^r \phi_{T_{n,i}}(\zeta_i),
\end{equation}
this would imply that \eqref{positive-2} holds, which would complete
the proof. Hence, we may assume that we have found a sequence $(\zeta_i)_{1\le i\le r}$ with
canonical local heights equal to 0. As before, we let the $(T_{n,i})_{\substack{n\in\mathbb{N}^{*}\\ 1\le i\le r}}$ be the corresponding sequence of polynomials for the $\zeta_i$, which replace the polynomials $P_{n,i}$.

Next we apply
the above process to another $w\in S_0\setminus\{\infty\}$ for which
there exists at least one $\zeta_i$ such that $\hhat_w(\zeta_i)>0$.
Note that when we apply the above process to the
$\zeta_1,\dots,\zeta_r$ at the place $w$, we might replace (at certain
steps of our process) the $\zeta_i$ by
\begin{equation}
\label{the form}
\sum_{j}\phi_{C_j}(\zeta_j)\in\Gamma.
\end{equation}
Because for the places $v\in S_0$ for which we already completed the above process, $\hhat_v(\zeta_i)=0$ for all $i$, then by the triangle inequality for the local height, we also
have
$$\hhat_v\left(\sum_j \phi_{C_j}(\zeta_j)\right)=0.$$

If we went through all $v\in S_0\setminus\{\infty\}$, and if the above
process did not yield that \eqref{positive-4} holds for some $v\in
S\setminus\{\infty\}$, then we are left with
$\zeta_1,\dots,\zeta_r\in\Gamma$ such that for all $i$ and all
$v\ne\infty$, we have $\hhat_v(\zeta_i)=0$.  Note that since
$\hhat_v(\zeta_i)=0$ for each $v\ne\infty$ and each
$i\in\{1,\dots,r\}$, then by the triangle inequality for local
heights, for all polynomials $Q_1,\dots,Q_r$, we have
\begin{equation}
\label{triangle inequality}
\hhat_v\left(\sum_{i=1}^r \phi_{Q_i}(\zeta_i)\right)=0\text{ for $v\ne\infty$.}
\end{equation}
Lemma~\ref{jumps at an infinite place} and \eqref{triangle inequality}
show that for all polynomials $Q_i$,
\begin{equation}
\label{the final jump}
D_{\infty}\cdot\hhat_{\infty}\left(\sum_{i=1}^r\phi_{Q_i}(\zeta_i)\right)\in\mathbb{N}.
\end{equation}

We repeat the above process, this time for $v=\infty$. As before, we
conclude that either\begin{equation}
\label{positive-infinity}
\limsup_{n\to\infty}\frac{\log |\sum_{i=1}^r\phi_{T_{n,i}}(\zeta_i)|_{\infty}}{\sum_{i=1}^r q^{d\deg T_{n,i}}}>0
\end{equation}
or we are able to replace the $\zeta_i$ by some other elements $\beta_i$ (which are of the form \eqref{the form}) such that
$$\sum_{i=1}^r\hhat_{\infty}(\beta_i) < \sum_{i=1}^r\hhat_{\infty}(\zeta_i).$$
Using \eqref{the final jump}, we conclude that
\begin{equation}
\label{the final jump 2}
\sum_{i=1}^r\hhat_{\infty}(\zeta_i) - \sum_{i=1}^r\hhat_{\infty}(\beta_i)\ge\frac{1}{D_{\infty}}.
\end{equation}
Therefore, after a finite number of steps this process of replacing
the $\zeta_i$ must end, and it cannot end with all the new $\beta_i$
having local canonical height $0$, because this would mean that all
$\beta_i$ are torsion (we already knew that for $v\ne\infty$, we have
$\hhat_v(\zeta_i)=0$, and so, by \eqref{triangle inequality},
$\hhat_v(\beta_i)=0$). Because the $\beta_i$ are nontrivial ``linear''
combinations (in the $\phi$-module $\Gamma$) of the $\gamma_i$ which
span a torsion-free $\phi$-module, we conclude that indeed, the
$\beta_i$ cannot be torsion points. Hence, our process ends with
proving \eqref{positive-infinity} which proves \eqref{positive-4}, and
so, it concludes the proof of our Proposition~\ref{sup}.
\end{proof}

\begin{remark}
  If there is more than one infinite place in $K$, then we cannot
  derive Lemma~\ref{jumps at an infinite place}, and in particular, we
  cannot derive \eqref{the final jump 2}. The idea is that in this
  case, for each nontorsion $\zeta$ which has its local canonical
  height equal to $0$ at finite places, we only know that there exists \emph{some}
  infinite place where its local canonical height has \emph{bounded} denominator. However,
  we do not know if that place is the one which we analyze at that
  particular moment in our process from the proof of
  Proposition~\ref{sup}. Hence, we would not necessarily be able to
  derive \eqref{the final jump 2}.
\end{remark}



Now we are ready to prove Theorem~\ref{Siegel}.
\begin{proof}[Proof of Theorem~\ref{Siegel}.]
  Let $(\gamma_i)_i$ be a finite set of generators of $\Gamma$ as a
  module over $A=\Fq[t]$. At the expense of replacing $S$ with a
  larger finite set of places of $K$, we may assume $S$ contains all the
  places $v \in M_K$ which satisfy at least one of the following
  properties:
\begin{enumerate}
\item[$1.$] $\hhat_v(\gamma_i) > 0$ for some $1\le i\le r$.
\item[$2.$] $|\gamma_i|_v > 1$ for some $1\le i\le r$.
\item[$3.$] $|\alpha|_v > 1$.
\item[$4.$] $\phi$ has bad reduction at $v$.
\end{enumerate}
Expanding the set $S$ leads only to (possible) extension of the set of
$S$-integral points in $\Gamma$ with respect to $\alpha$. Clearly, for
every $\gamma\in \Gamma$, and for every $v\notin S$ we have $|\gamma|_v\le 1$. 
Therefore, using $3.$, we obtain
\begin{equation}
\label{restating S-integrality}
\begin{split}
  \gamma\in\Gamma\text{ is $S$-integral with respect to $\alpha$} 
  \Longleftrightarrow |\gamma-\alpha|_v=1\text{ for all $v\in
    M_K\setminus S$}.
\end{split}
\end{equation}

Moreover, using $1.$ from above, we conclude that for every $\gamma\in\Gamma$, and for
every $v\notin S$, we have $\hhat_v(\gamma)=0$ (see the proof of Fact~\ref{S_0 is finite}).

Next we observe that it suffices to prove Theorem~\ref{Siegel} under
the assumption that $\Gamma$ is a free $\phi$-submodule. Indeed,
because $A=\Fq[t]$ is a principal ideal domain, $\Gamma$ is a direct
sum of its finite torsion submodule $\Gamma_{\tor}$ and a free
$\phi$-submodule $\Gamma_1$ of rank $r$, say. Therefore,
$$\Gamma = \bigcup_{\gamma\in\Gamma_{\tor}} \gamma+\Gamma_1.$$
If we
show that for every $\gamma_0\in\Gamma_{\tor}$ there are finitely many
$\gamma_1\in\Gamma_1$ such that $\gamma_1$ is $S$-integral with
respect to $\alpha-\gamma_0$, then we obtain the conclusion of
Theorem~\ref{Siegel} for $\Gamma$ and $\alpha$ (see \eqref{restating S-integrality}).

Thus from now on, we assume $\Gamma$ is a free $\phi$-submodule of
rank $r$. Let $\gamma_1,\dots,\gamma_r$ be a basis for $\Gamma$ as an
$\Fq[t]$-module. We reason by contradiction. Let
$$\sum_{i=1}^r \phi_{P_{n,i}}(\gamma_i) \in\Gamma$$
be an infinite sequence of elements $S$-integral with respect to $\alpha$. Because of the $S$-integrality assumption (along with the assumptions
on $S$), we conclude that for every
$v \notin S$, and for every $n$ we have
$$ \frac{ \log
  |\sum_{i=1}^r \phi_{P_{n,i}}(\gamma_i) - \alpha |_v}{\sum_{i=1}^r
  q^{d\deg P_{n,i}}} = 0.$$
Thus, using the product formula, we see that
\begin{equation*}
\begin{split}
 & \limsup_{n\to\infty} \sum_{v\in S} \frac{ \log
  |\sum_{i=1}^r \phi_{P_{n,i}}(\gamma_i) - \alpha |_v}{\sum_{i=1}^r
  q^{d\deg P_{n,i}}} \\
& =  \limsup_{n\to\infty} \sum_{v\in M_K} \frac{ \log
  |\sum_{i=1}^r \phi_{P_{n,i}}(\gamma_i) - \alpha |_v}{\sum_{i=1}^r
  q^{d\deg P_{n,i}}}\\
& = 0.
\end{split}
\end{equation*}

On the other hand, by Proposition \ref{sup}, there is some place $w\in S$
and some number $\delta > 0$ such that
$$
\limsup_{n\to\infty} \frac{ \log |\sum_{i=1}^r
  \phi_{P_{n,i}}(\gamma_i) - \alpha |_w}{\sum_{i=1}^r q^{d\deg
    P_{n,i}}} = \delta > 0.$$
So, using Lemma \ref{inf}, we see that
\begin{equation*}
\begin{split}
& \limsup_{n\to\infty} \sum_{v\in S} \frac{ \log
  |\sum_{i=1}^r \phi_{P_{n,i}}(\gamma_i) - \alpha |_v}{\sum_{i=1}^r
  q^{d\deg P_{n,i}}} \\
& \ge \sum_{\substack{v\in S\\ v\ne w}}\liminf_{n\to\infty}\frac{ \log
  |\sum_{i=1}^r \phi_{P_{n,i}}(\gamma_i) - \alpha |_v}{\sum_{i=1}^r
  q^{d\deg P_{n,i}}} +   \limsup_{n\to\infty} \frac{ \log |\sum_{i=1}^r
  \phi_{P_{n,i}}(\gamma_i) - \alpha |_w}{\sum_{i=1}^r q^{d\deg
    P_{n,i}}} \\ 
& \geq 0 + \delta \\
& > 0.\\
\end{split}
\end{equation*}

Thus, we have a contradiction which shows that there cannot be infinitely many
elements of $\Gamma$ which are $S$-integral for $\alpha$. 
\end{proof}


\def\cprime{$'$} \def\cprime{$'$} \def\cprime{$'$} \def\cprime{$'$}
\providecommand{\bysame}{\leavevmode\hbox to3em{\hrulefill}\thinspace}
\providecommand{\MR}{\relax\ifhmode\unskip\space\fi MR }
\providecommand{\MRhref}[2]{%
  \href{http://www.ams.org/mathscinet-getitem?mr=#1}{#2}
}
\providecommand{\href}[2]{#2}

\end{document}